\documentclass[a4paper,10pt,reqno]{amsart}

\usepackage[top=2.5cm, bottom=2.5cm, left=2.5cm, right=2.5cm]{geometry}
\usepackage{amssymb,amsmath,amsthm,ascmac,array,bm,multirow}
\usepackage{booktabs}
\usepackage[dvipdfmx]{graphicx}
\usepackage[subrefformat=parens,font=scriptsize,labelformat=parens]{subcaption}
\usepackage{xcolor}
\usepackage{here}

\allowdisplaybreaks
\numberwithin{equation}{section}

\newtheorem{thm}{Theorem}[section]

\newtheorem{rem}[thm]{Remark}

\newcommand{\mcc}{\mathcal{C}}

\newcommand{\mci}{\mathcal{I}}

\newcommand{\mcl}{\mathcal{L}}

\newcommand{\mbbr}{\mathbb{R}}

\newcommand{\al}{\alpha}
\newcommand{\del}{\delta}

\newcommand{\sig}{\sigma}
\newcommand{\ep}{\epsilon}
\newcommand{\D}{\Delta}

\newcommand{\gam}{\gamma}

\newcommand{\p}{\partial}

\newcommand{\cil}{\xrightarrow{\mcl}} 

\newcommand{\diag}{\mathop{\rm diag}} 
\newcommand{\trace}{\mathop{\rm trace}} 

\def\ds#1{\displaystyle{#1}} 

\def\nn{\nonumber}

\def\wp{Wiener process}
\def\lp{L\'{e}vy process}

\def\sumi{\sum_{i=1}^{N}}

\newcommand{\pr}{P} \newcommand{\E}{E}
\newcommand{\var}{\mathrm{var}}
\newcommand{\cov}{\mathrm{cov}}

\def\rev#1{\textcolor{black}{#1}}

\newcommand{\tz}{\theta_{0}}
\newcommand{\tes}{\hat{\theta}_{N}}
\newcommand{\bes}{\hat{\beta}_{N}}
\newcommand{\ves}{\hat{v}_{N}}

\title[On local likelihood asymptotics for Gaussian mixed-effects model]
{On local likelihood asymptotics for Gaussian mixed-effects model with system noise}

\author[T. Imamura]{Takumi Imamura}
\address{Biostatistics Center, Shionogi \& Co., Ltd., Osaka, Japan}

\author[H. Masuda]{Hiroki Masuda}
\address{
Graduate School of Mathematical Sciences, University of Tokyo, 3-8-1 Komaba Meguro-ku Tokyo 153-8914, Japan.
}
\email{hmasuda@ms.u-tokyo.ac.jp}

\author[H. Tajima]{Hayato Tajima}
\address{
Graduate School of Mathematics, Kyushu University, 744 Motooka Nishi-ku Fukuoka 819-0395, Japan.
}

\date{\today}

\begin{document}
\setlength{\baselineskip}{4.5mm}

\maketitle

\begin{abstract}
The Gaussian mixed-effects model driven by a stationary integrated Ornstein-Uhlenbeck process has been used for analyzing longitudinal data having an explicit and simple serial-correlation structure in each individual. However, the theoretical aspect of its asymptotic inference is yet to be elucidated.
We prove the local asymptotics for the associated log-likelihood function, which in particular guarantees the asymptotic optimality of the suitably chosen maximum-likelihood estimator.
We illustrate the obtained asymptotic normality result through some simulations for both balanced and unbalanced datasets.
\end{abstract}

\section{Introduction}

\subsection{Setup and objective}

We consider the local likelihood asymptotics for the Gaussian linear mixed-effects integrated Ornstein-Uhlenbeck (IOU) model originally introduced in \cite{TayCumSy94}, in which the dynamics of the $i$th individual is described by
\begin{equation}
Y_i(t) = X_i(t)^{\rev{\top}}\beta +Z_i(t)^{\rev{\top}}b_i+ W_i(t) +\epsilon_i(t)
\label{hm:model}
\end{equation}
for $i=1,\cdots,N$ and a given fixed time horizon $t\in[0,T]$, where the ingredients are given as follows.
\begin{itemize}
\item We observe $\{(t_{ij},X_i(t_{ij}), Y_i(t_{ij}), Z_i(t_{ij}))\}_{j=1}^{n_i}$ for each $i=1,\cdots,N$, with
\begin{equation}
\sup_N \max_{i\leq N} n_i < \infty,
\nn
\end{equation}
where $0\equiv t_{i0}<t_{i1}<\dots<t_{i n_i}\le T$ for each $i$, and where $X_i(t)\in\mbbr^{p_\beta}$ and $Z_i(t)\in\mbbr^{p_b}$ are non-random explanatory variables (processes) satisfying that
\begin{equation}
\rev{
\sup_{t\in[0,T]} \max_{i\leq N} \left( |X_i(t)|\vee |Z_i(t)| \right) = O(1).
}
\nn
\end{equation}
Here and in what follows, the order and asymptotic symbols are used for $N\to\infty$ unless otherwise mentioned.

\item $\beta\in\mbbr^{p_\beta}$ is fixed-effect unknown parameter, and the random effects $b_1,b_2,\dots \in \mbbr^{p_b}$ are i.i.d. $N_{p_b}(0,G(\gam))$ for some \rev{known} function $G(\gam):\,\mbbr^{p_\gam}\to\mbbr^{p_b}\otimes\mbbr^{p_b}$.

\item $W_i(t)$ is a system noise described by the i.i.d. (centered) integrated Gaussian Ornstein-Uhlenbeck process
\begin{equation}
W_i(t) = \int_0^t \zeta_i(s) ds
= \frac{\zeta_i(0)}{\al}(1-e^{-\al t}) + \frac{\tau}{\al} \int_0^t (1-e^{-\al(t-v)}) dw_i(v)
\label{hm:int-ou.sol}
\end{equation}
for $\zeta_1,\zeta_2,\dots$ being i.i.d. stationary Gaussian OU process of the form
\begin{equation}
\zeta_i(t) =\int_{-\infty}^t e^{-\al(t-s)}\tau dw_i(s) \sim N\left(0,\,\frac{\tau^2}{2\al}\right),
\label{hm:sOU_def}
\end{equation}
which equals a solution process to the stochastic differential equation
\begin{equation}
d\zeta_i(t)=-\al \zeta_i(t) dt+\tau dw_i(t),
\nonumber
\end{equation}
with $w_1,w_2,\dots$ denoting i.i.d. standard {\wp es} and both $\al>0$ and $\tau>0$ being unknown parameters.

\item $\ep_i(t)$ denotes a measurement error at time $t$ for the individual $i$. We assume that $\ep_1,\ep_2,\dots$ are independent centered 
Gaussian white noises such that for each $i$, $\cov_\theta[\ep_i(t_{ij}),\ep_i(t_{ik})]=\sig^2 \del_{jk}$ ($\del_{jk}$ denotes the Kronecker delta).

\item $\{b_i\}$, $\{W_i\}$, and $\{\ep_i\}$ are \rev{mutually independent of one another}.

\end{itemize}
The model allows us to look at irregularly spaced and different-number observations across individuals and also missing values for some of the variables;
use of a Gaussian process for such simple and transparent correlation-structure modeling goes back to \cite{Dig88}.
The time-integrated character is suitable for many applications where sample paths of system noise, hence those of the objective time series, are smoother than non-differentiable diffusion-type models; as we will briefly mention in Remark \ref{hm:rem_model.extension}, our asymptotic framework could handle several other processes for $W_i$.
The covariance structure of the system noise is expressed only through the two parameters $\al$ and $\tau$ in a unified way.
From \eqref{hm:intou.cov} below, we see that the process $W_i$ approximates: for $j\ne k$, with $c:=\tau/\al>0$ being fixed,
\begin{itemize}
\item A scaled Wiener process where $H_{i;jk}(\al,\tau)\approx c^2 \min(t_{ij},t_{ik})$ for $\al\to\infty$;
\item A Gaussian white noise process where $H_{i;jk}(\al,\tau)\approx 0$ for $\al\to 0$.
\end{itemize}
In the context of the random effect models, the parameter $\al>0$ is referred to as ``the degree of derivative tracking'' (a degree of maintaining the same trajectory over time):
the process $W_i$ becomes degenerate (to the process identically zero) for $\al\to 0$ with fixed $c>0$ (then $\tau\to 0$), as can be seen from the expression \eqref{hm:sOU_def}.

Our objective is to study the local asymptotic property of the associated likelihood function for estimating the finite-dimensional parameter
\begin{equation}
\nn
\theta := \left(\beta, \gam, \al,\tau,\sig^2 \right) 
\in \mbbr^{p_\beta}\times\mbbr^{p_\gam}\times(0,\infty)^3 \subset\mbbr^p,
\end{equation}
where $p:=p_\beta + p_\gam + 3 \ge 5$ denotes the dimension of $\theta$.
The statement is given in Section \ref{hm:sec_main}.
We will present some numerical experiments in Section \ref{hm:sec_simulations}.

\subsection{Some literature review in brief}

After \cite{TayCumSy94} introduced the model \eqref{hm:model}, several works studied its application to specific areas\rev{, for example HIV study for clinical markers of AIDS \cite{BosTayLaw1998}}.
Then, \cite{SyTayCum97} extended the model to the case of bivariate response and applied it to analyzing AIDS data, in particular predictions of future observations and cause-and-effect relationship therein.
Later on, \cite{ZhaLinRazSow98} introduced a semiparametric extension, by adding a nonparametric mean function of time.
More recently, \cite{HugKenSteTil17} developed an optimization algorithm and \cite{hughes2017_stata} did \texttt{xtmixediou} using Stata's matrix programming language.
As a further application example, we refer to \cite{GSACZS22}, where the authors studied Bayesian regularization of a related model and applied it to analyzing CD4 yeast cell-cycle genomic data: Bayesian ridge, lasso, and elastic net approaches were considered together with computational aspects of the posterior.

Nevertheless, a related theoretical study from an asymptotic viewpoint seems missing in the literature.
The primary scope of this paper is to derive the local asymptotics for the maximum-likelihood estimator (MLE), providing us with the fundamental notion of an asymptotically efficient estimator.
See also Remarks 
\ref{hm:rem_LAN}, \ref{hm:rem_future.work}, and \ref{hm:rem_model.extension} for some related details and issues.

\section{Local asymptotics}
\label{hm:sec_main}

For notational convenience, let us write
\begin{equation}
\xi_{ij} = \xi_i(t_{ij})
\nonumber
\end{equation}
for $\xi=X$, $Y$, $Z$, $W$, and $\ep$.
Further, let $X_i=(X_{ij}) \in \mbbr^{n_i}\otimes\mbbr^{p_\beta}$, $Y_i=(Y_{ij}) \in \mbbr^{n_i}$, $Z_i=(Z_{ij}) \in \mbbr^{n_i}\otimes\mbbr^{p_b}$, $W_i=(W_{ij}) \in \mbbr^{n_i}$, and $\ep_i=(\ep_{ij}) \in \mbbr^{n_i}$.
With these shorthands and \eqref{hm:model}, we have the expression
\begin{equation}
Y_i = X_i\beta +Z_ib_i+ W_i +\epsilon_i
\nn
\end{equation}
for the sample from the $i$th individual.
We denote by
\begin{equation}
v=(\gam, \al, \tau, \sig^2)
\nonumber
\end{equation}
the parameters contained in the covariance matrix of $Y_i$; $\gam_k$ and $v_l$ denote the $k$th and $l$th components of $\gam$ and $v$, respectively.
Let $\pr_{\theta}$ denote the distribution of $(\{b_i\},\{W_i\},\{\ep_i\})$, and write $\E_\theta$ and $\cov_\theta$ for the associate expectation and covariance, respectively.

The process $W_i$ is centered in the sense that $\E_\theta[W_i(t)]=0$ for each $t$. 
By the expression \eqref{hm:int-ou.sol} 
\rev{and the stationarity of $\zeta_i(\cdot)$}, 
we obtain the following specific covariance structure $H_i(\al,\tau)=:(H_{i;jk}(\al,\tau))_{j,k}$:
\begin{align}
H_{i;jk}(\al,\tau) &:= \cov_\theta\left[W_{ij}, W_{ik}\right]\nn \\
&=\frac{1}{\al^2}(1-e^{-\al t_{ij}})(1-e^{-\al t_{ik}}) \E_\theta[\zeta_i(0)^2] \nn\\
&{}\qquad + \frac{\tau^2}{\al^2}\int_0^{t_{ij}\wedge t_{ik}} (1-e^{-\al (t_{ij}-s)})(1-e^{-\al (t_{ik}-s)}) ds
\label{hm:intou.cov+}\\
&=\frac{\tau^2}{2\al^3}\left(2\al \min(t_{ij},t_{ik})+e^{-\al t_{ij}}+e^{-\al t_{ik}}-1-e^{-\al |t_{ij}-t_{ik}|}\right).
\label{hm:intou.cov}
\end{align}
We have (under $\pr_\theta$)
\begin{equation}
Y_i \overset{\pr_\theta}{\sim} N_{n_i}\left( X_i\beta,\, Q_i(v)\right)
\nonumber
\end{equation}
for $i=1,\dots,N$, where
\begin{equation}
Q_i(v) := Z_i G(\gamma)Z_i^\top + H_i(\alpha,\tau) + \sigma^2 I_{n_i},
\label{hm:Q_def}
\end{equation}
with $I_{p}$ denoting the $p$-dimensional identity matrix.
Here and in what follows, we set the parameter space to be 
\begin{equation}
\Theta = \Theta_\beta \times \Theta_v= \Theta_\beta \times \Theta_\gam \times \Theta_{(\al,\tau,\sig^2)}\subset \mbbr^{p_\beta}\times\mbbr^{p_\gam}\times(0,\infty)^3,
\nonumber
\end{equation}
a domain in $\mbbr^p$, for which the covariances $Q_i(v)$ are uniformly non-degenerate:
\begin{equation}
\forall v\in \Theta_v,\quad 
\rev{
\inf_N \inf_{1\le i\le N}\lambda_{\min}\left(Q_i(v)\right) > 0.
}
\label{hm:A_Q>0}
\end{equation}
Then, the log-likelihood function is well-defined for $\theta\in\Theta$ and is given by
\begin{align}
\ell_N(\theta) &= \sumi \log \phi_{n_i}\left( Y_i;\,X_i\beta,\, Q_i(v)\right)
\nn\\
&=-\frac{\log(2\pi)}{2}\sum_{i=1}^N n_i-\frac{1}{2}\sum_{i=1}^N\left\{\log\left|Q_i\left(v\right)\right|
+\left(Y_i-X_i\beta\right)^\top Q_i\left(v\right)^{-1}\left(Y_i-X_i\beta\right)\right\}.
\label{hm:log-LF}
\end{align}

\medskip

We write
\begin{equation}
\D_N(\theta) = \frac{1}{\sqrt{N}} \p_{\theta}\ell_{N}(\theta)
\nonumber
\end{equation}
for the normalized score function, where $\p_\theta$ denotes the partial-differentiation operator with respect to $\theta$.
For a multilinear form $M=\{M_{i_1 i_2 \dots i_m}\}$, we will write $M[u_{i_1},\dots,u_{i_m}]=\sum_{i_1,\dots,i_m} M_{i_1 i_2 \dots i_m} u_{i_1}\dots u_{i_m}$.

\begin{thm}
\label{hm:thm_local.LF.asymp}
Fix any $\theta_0 = (\beta_0,v_0) =(\beta_0, \gamma_0, \alpha_0, \tau_0 , \sigma^2_0) \in \Theta$ as a true value of $\theta$.
Suppose the following conditions \rev{hold}:
\begin{itemize}
\item The function $G(\gam):\, \Theta_\gam \to\mbbr^{p_b}\otimes\mbbr^{p_b}$ is of class $\mcc^3(\overline{\Theta_\gam})$.

\item There exist symmetric-matrix-valued $\mcc^1(\overline{\Theta_v})$-functions $A(v)$ and $U(v)=(U_{jk}(v))_{j,k}$ satisfying that for each $v\in\Theta_v$ \rev{and all $j, k = 1, \dots, p_\gam +3$}, 
\begin{align}
& \frac{1}{N}\sum_{i=1}^N X_i^\top Q_i(v)^{-1}X_i \to A(v),
\label{hm:thm_A1}\\
& \frac1N \sumi 
\frac{1}{2}\trace\left\{Q_i(v)^{-1}\left(\partial_{v_k}Q_i(v)\right)Q_i(v)^{-1}\left(\partial_{v_j}Q_i(v)\right)\right\}
\to U_{jk}(v),
\label{hm:thm_A2}
\end{align}
and that both $A(v)$ and $U(v)$ are positive-definite uniformly in $v$ oevr each compact $K_v\subset\Theta_v$.
\end{itemize}
Then, the following statements hold under $\pr_{\tz}$.
\begin{enumerate}
\item 
For any bounded sequence $(u_N)_{N\ge 1}\subset\mbbr^p$,
\begin{equation}
\ell_{N}\left(\tz+\frac{1}{\sqrt{N}}u_N\right) - \ell_{N}\left(\tz\right) 
= \D_N(\tz)[u_N] - \frac{1}{2} \mci(v_0) [u_N^{\otimes 2}] + o_{p}(1),
\nn
\end{equation}
where $\D_N(\tz) \cil N_p(0,\mci(v_0))$ and
\begin{equation}
\mci(v) := \diag\left(A(v), ~U(v)\right).
\nonumber
\end{equation}

\item 
There exists a local maximum point $\tes$ of $\ell_{n}(\theta)$ with $\pr_{\tz}$-probability tending to $1$, for which
\begin{equation}
\sqrt{N}(\tes -\tz) = \mci(v_0)^{-1}\D_N(\tz) + o_{p}(1) \cil N_{p}\left(0,\, \mci(v_0)^{-1} \right).
\nonumber
\end{equation}
\end{enumerate}
\end{thm}

Before the proof, we make a couple of remarks.


\begin{rem}\normalfont
Since the terminal time $T>0$ is fixed throughout, it is not essential that $\al>0$ for obtaining the local asymptotic results in Theorem \ref{hm:thm_local.LF.asymp}; even when $\al\le 0$, the covariance matrix $H_i(\al,\tau)$ is well-defined by \eqref{hm:intou.cov+}, and then \eqref{hm:A_Q>0} remains valid. 
However, it should be noted that the mean-reverting feature of the process $\zeta_i(t)$ only holds for $\al>0$ and the expression \eqref{hm:intou.cov} is based on \eqref{hm:sOU_def}.
\end{rem}

\begin{rem}[Asymptotic efficiency]\normalfont
\label{hm:rem_LAN}
There are several important implications and consequences of Theorem \ref{hm:thm_local.LF.asymp} worth mentioning.
Theorem \ref{hm:thm_local.LF.asymp}(1) shows the local asymptotic normality (LAN) of the family $\{\pr_\theta\}_{\theta\in\Theta}$, based on which the classical asymptotic theory enables us to define the asymptotic efficiency of regular estimators.
\rev{
Below we give a brief account; among others, we refer to \cite[Chapter 2]{BasSco83} and \cite[Section 3]{Jeg82} for general accounts.
In our context, ``any'' estimators $\tes^\ast$ satisfying that
\begin{equation}
\sqrt{N}(\tes^\ast -\tz) = \mci(v_0)^{-1}\D_N(\tz) + o_{p}(1)
\label{def:ACe}
\end{equation}
are regular. Hence, the LAN property implies that the Haj\'{e}k-Le Cam asymptotic lower bound for the quadratic loss functions is in force:
\begin{equation}
\liminf_{n\to\infty} \E_{\tz}\left[ \left| 
\sqrt{N}(\tes^\ast - \tz) \right|^2 \right]
\ge \int \big|\mci(v_0)^{-1/2}z\big|^2 \phi(z)dz = \trace\left(\mci(v_0)^{-1}\right),
\label{lower.bound}
\end{equation}
where $\phi(z)$ denotes the density of $N_p(0,I_p)$.
Thus, under the LAN property, we may call any regular estimator $\tes^\ast$ satisfying \eqref{def:ACe} asymptotically efficient;
the terminology ``efficient'' stems from the minimality of the asymptotic covariance, and also from the asymptotically maximal concentration.
In particular, Theorem \ref{hm:thm_local.LF.asymp} ensures both \eqref{def:ACe} and \eqref{lower.bound} with $\tes^\ast=\tes$, hence $\tes$ is asymptotically efficient in the above sense. Since \eqref{def:ACe} entails the asymptotic normality $\sqrt{N}(\tes^\ast -\tz) \cil N_{p}\left(0,\, \mci(v_0)^{-1} \right)$, it is worth noting that the convergence of moments $\E_{\tz}[\{\sqrt{N}(\tes -\tz)\}^{\otimes 2}] \to \mci(v_0)^{-1}$ holds as soon as the sequence $\{|\sqrt{N}(\tes -\tz)|^2\}_N$ is uniformly integrable.
We also note that the following Studentized version for any asymptotically efficient estimator $\tes^\ast=(\bes^\ast,\ves^\ast)$ can be easily derived:
\begin{equation}
\diag\big(\hat{A}_N, ~\hat{U}_N\big)^{1/2} \sqrt{N} (\tes^\ast -\tz) \cil N_{p}(0, I_p),
\nonumber
\end{equation}
with
\begin{align}
\hat{A}_N &:= \frac{1}{N}\sum_{i=1}^N X_i^\top Q_i(\ves^\ast)^{-1}X_i,
\nn\\
\hat{U}_N &:= \frac1N \sumi 
\frac{1}{2}\trace\left\{Q_i(\ves^\ast)^{-1}\left(\partial_{v}Q_i(\ves^\ast)\right)Q_i(\ves^\ast)^{-1}\left(\partial_{v}Q_i(\ves^\ast)\right)\right\},
\nonumber
\end{align}
where the elements of $\hat{U}_N$ are more specified by $\partial_{\gam}Q_i(v) = Z_i \p_\gam G(\gam) Z_i^\top$, $\partial_{(\al,\tau)}Q_i(v) = \p_{(\al,\tau)}H_i(\al,\tau)$, and $\partial_{\sig^2}Q_i(v) = I_{n_i}$; recall \eqref{hm:Q_def} and $v=(\gam,\al,\tau,\sig^2)$.
}
\end{rem}

\begin{rem}[Theoretical refinments]\normalfont
\label{hm:rem_future.work}
A good root $\tes$ of the likelihood equation $\p_\theta\ell_N(\theta)=0$ in Theorem \ref{hm:thm_local.LF.asymp} is not a single choice and also may not necessarily be the best from a computational point of view. 
Our interest here is in the first-order asymptotic inference, and we did not consider the conventional REML (restricted maximum-likelihood) estimator.
Likewise, a popular way of construction of such an estimator $\tes^\ast$ includes the stepwise one: usually, one first uses the (un-weighted) least-squares for $\beta$, and then proceeds with a variance-component estimation; note that we then obtain a globally consistent estimator as was studied in \cite{Taj21_m.thesis} for joint estimation of all the components of $\theta$ under a series of regularity conditions.
We could derive the asymptotic distribution of the above-mentioned stepwise estimator even when the sources of randomness in the model are non-Gaussian, in particular even when the driving process in $W_i(t)$ is a non-Gaussian {\lp}; this point will have an attractive feature, for it ensures that the inference procedure becomes robust against the misspecified Gaussian assumptions. Moreover, it would be possible to deduce the uniform tail-probability estimate of the associated Gaussian quasi-maximum likelihood estimator, enabling us to conclude the asymptotic efficiency in the sense mentioned in Remark \ref{hm:rem_LAN}, and further study the model selection issue by constructing appropriate information criteria.
We will report the related details elsewhere.
\end{rem}

\begin{rem}[Other system-noise processes]\normalfont
\label{hm:rem_model.extension}
Although we are focusing on the IOU process for $W_i$, the same likelihood analysis based on the low-frequency sampling for each individual could formally go through other system-noise processes parametrized by a finite-dimensional parameter as long as $\cov_\theta[W_i(s),W_i(t)]$ exist and is explicitly given.
For example, $W_1,\dots,W_{N}$ could be i.i.d. copies of a (drift-free) scaled fractional Brownian motion:
$W_i$ is a centered Gaussian process with stationary increments such that $\E_\theta[W_i(t)]=0$, $\var_\theta[W_i(t)-W_i(s)]=\tau^2|t-s|^{2\mathsf{H}}$, and
\begin{equation}
\cov_\theta[W_i(t),W_i(s)] = \frac{\tau^2}{2}\left( t^{2\mathsf{H}} + s^{2\mathsf{H}} - |t-s|^{2\mathsf{H}}\right),\qquad t,s\ge 0,
\nonumber
\end{equation}
for the scale parameter $\tau>0$ and the Hurst parameter $\mathsf{H}\in(0,1)$;
then, the covariance-matrix parameter is $v=(\gam, \mathsf{H}, \tau, \sig^2)$ and \eqref{hm:intou.cov} becomes $H_i'(\mathsf{H},\tau)=:(H_{i;jk}'(\mathsf{H},\tau))_{j,k}$ with
\begin{equation}
H_{i;jk}'(\mathsf{H},\tau) := \frac{\tau^2}{2}\left( t_{ij}^{2\mathsf{H}} + t_{ik}^{2\mathsf{H}} - |t_{ij} - t_{ik}|^{2\mathsf{H}}\right).
\nonumber
\end{equation}
Correspondingly, we could deduce a variant of Theorem \ref{hm:thm_local.LF.asymp} without any essential change:
we replace $Q_i(v)$ by $Q'_i(v) := Z_i G(\gamma)Z_i^\top + H_i'(\mathsf{H},\tau) + \sigma^2 I_{n_i}$, 
and impose similar assumptions to \eqref{hm:thm_A1} and \eqref{hm:thm_A2}; for the latter, the partial derivative $\p_{(\mathsf{H},\tau)}Q'_i(v)$ are given through
\begin{align}
\p_{\mathsf{H}}H_i'(\mathsf{H},\tau) &= \tau^2 \left( t_{ij}^{2\mathsf{H}}\log(t_{ij}) + t_{ik}^{2\mathsf{H}}\log(t_{ik}) - |t_{ij}-t_{ik}|^{2\mathsf{H}}\log|t_{ij}-t_{ik}| \right),
\nn\\
\p_{\tau}H_i'(\mathsf{H},\tau) &= \tau \left( t_{ij}^{2\mathsf{H}} + t_{ik}^{2\mathsf{H}} - |t_{ij} - t_{ik}|^{2\mathsf{H}}\right).
\nonumber
\end{align}
As a specific application to longitudinal biomedical data, this model was used in \cite{fBm_ME} for empirical analysis of CD4 counts in HIV-positive patients.
Compared with the IOU model, however, the fractional Brownian motion cannot quantitatively capture the degree of derivative tracking.
\end{rem}

\medskip

\begin{proof}[Proof of Theorem \ref{hm:thm_local.LF.asymp}]

We introduce the normalized observed information matrix:
\begin{equation}
\mci_N(\theta) := - \frac{1}{N}\p_{\theta}^{2}\ell_N(\theta).
\nonumber
\end{equation}
We are going to complete the proof by verifying the following two conditions for $N\to\infty$: for any $\ep>0$, $c>0$, and compact $K\subset\Theta$,
\begin{align}
S_{1,N}(\ep,K) &:= \sup_{\theta\in K}\pr_\theta\left[
\left|\mci_N(\theta) - \mci(v)\right| > \ep \right] \to 0,
\label{C1}\\
S_{2,N}(\ep,c,K) &:= \sup_{\theta\in K}\pr_\theta\left[
\frac{1}{\sqrt{N}}\sup_{\theta'\in\Theta:\,|\theta'-\theta|\le c N^{-1/2}}\left| \p_{\theta}\mci_N(\theta') \right| > \ep
\right] \to 0.
\label{C2(ii)}
\end{align}
Using the criterion in \cite{Swe80} (see Theorems 1 and 2 therein), these conditions ensure both claims in Theorem \ref{hm:thm_local.LF.asymp}.

\medskip

To prove the law of large numbers \eqref{C1}, we recall the expression \eqref{hm:log-LF} of the log-likelihood function $\ell_N(\theta)$.
To proceed, we need to compute the partial derivatives of $\ell_N(\theta)$.
Let $H_i^{(\al)}:=\p_\al H_i$ and $H_i^{(\tau)}:=\p_\tau H_i$. By \eqref{hm:intou.cov}, the $(j,k)$th entries of these matrices are given as follows:
\begin{equation*}
\begin{split}
H_i^{(\al)}(\al,\tau)_{j,k}
	&=\frac{\tau^2}{2\alpha^4}\Big( -4\alpha\min\left(t_{ij},t_{ik}\right)-\left(3+\alpha t_{ij}\right)e^{-\alpha t_{ij}}-\left(3+\alpha t_{ik}\right)e^{-\alpha t_{ik}}+3
	\nn\\
	&{}\qquad +\left(3+\alpha|t_{ij}-t_{ik}|\right)e^{-\alpha |t_{ij}-t_{ik}|}\Big),\\
H_i^{(\tau)}(\al,\tau)_{j,k}
	&=\frac{\tau}{\alpha^3}\left(2\alpha \min\left(t_{ij},t_{ik}\right)+e^{-\alpha t_{ij}}+e^{-\alpha t_{ik}}-1-e^{-\alpha |t_{ij}-t_{ik}|}\right).
\end{split}
\end{equation*}
Then, we have the expressions for the first-order derivatives:
\begin{align*}
\p_\beta\ell_N(\theta) &=\sum_{i=1}^N\left\{X_i^\top Q_i\left(v\right)^{-1}Y_i-X_i^\top Q_i\left(v\right)^{-1}X_i\beta\right\},
\nn\\
\partial_{ \gamma_l} \ell_N(\theta)&=\frac{1}{2}\sumi \Big\{\left(Y_i-X_i\beta\right)^\top Q_i(v)^{-1}Z_i\left(\partial_{\gamma_l} G(\gamma)\right)Z_i^\top Q_i(v)^{-1}\left(Y_i-X_i\beta\right) \nn\\
&{}\qquad -\trace\left(Q_i(v)^{-1}Z_i\left(\partial_{\gamma_l} G(\gamma)\right)Z_i^\top \right)\Big\},\\
\partial_{\alpha} \ell_N(\theta)&=\frac{1}{2}\sum_{i=1}^N\left\{\left(Y_i-X_i\beta\right)^\top Q_i(v)^{-1}H_i^{(\alpha)}(\alpha,\tau)Q_i(v)^{-1}\left(Y_i-X_i\beta\right)-\trace\left(Q_i(v)^{-1}H_i^{(\alpha)}(\alpha,\tau)\right)\right\},\\
\partial_{\tau} \ell_N(\theta)&=\frac{1}{2}\sum_{i=1}^N\left\{\left(Y_i-X_i\beta\right)^\top Q_i(v)^{-1}H_i^{(\tau)}(\alpha,\tau)Q_i(v)^{-1}\left(Y_i-X_i\beta\right)-\mathrm{trace}\left(Q_i(v)^{-1}H_i^{(\tau)}(\alpha,\tau)\right)\right\},\\
\partial_{\sigma^2} \ell_N(\theta)&=\frac{1}{2}\sum_{i=1}^N\left\{\left(Y_i-X_i\beta\right)^\top \left(Q_i(v)^{-1}\right)^2\left(Y_i-X_i\beta\right)-\mathrm{trace}\left(Q_i(v)^{-1}\right)\right\}.
\end{align*}

Then, we obtain the expressions for the second-order derivatives:
\begin{align}
\partial^2_\beta \ell_N(\theta) &= -\sum_{i=1}^N X_i^\top Q_i(v)^{-1}X_i,
\nn\\
\partial_{\beta}\p_{v_k} \ell_N(\theta) &= \sum_{i=1}^N X_i^\top Q_i(v)^{-1}\left(\partial_{v_k}Q_i(v)\right)Q_i(v)^{-1}\left(Y_i-X_i\beta\right),
\nn\\
\partial_{v_j}\p_{v_k} \ell_N(\theta) &= \frac{1}{2}\sum_{i=1}^N\Big\{
\left(Y_i-X_i\beta\right)^\top \partial_{v_j}\left\{Q_i(v)^{-1}\left(\partial_{v_k} Q_i(v)\right)Q_i(v)^{-1}\right\}\left(Y_i-X_i\beta\right)
\nn\\
&{}\qquad -\partial_{v_j}\left\{\mathrm{trace}\left(Q_i\left(v\right)^{-1}\left(\partial_{v_k} Q_i(v)\right)\right)\right\}\Big\}.
\nonumber
\end{align}

First, by \eqref{hm:thm_A1} we have (deterministic convergence)
\begin{equation}
-\frac{1}{N}\partial^2_\beta \ell_N\left(\theta\right)
=\frac{1}{N}\sum_{i=1}^NX_i^\top Q_i(v)^{-1}X_i \to A(v)
\nonumber
\end{equation}
for each $\theta$; under \eqref{hm:A_Q>0}, this is valid uniformly in $\theta\in K$ since the derivative $\p_v\{Q_i(v)^{-1}\}$ is bounded over $K$.
Next, since the summands of $\partial_{\beta}\p_{v_k} \ell_N(\theta)$ is $\E_\theta$-expectation zero for each $\theta$ and since $Y_i - X_i\beta \overset{\pr_\theta}{\sim} N_{n_i}\left( 0,\, Q_i(v)\right)$, \rev{by Burkholder's inequality}, we have
\begin{align}
& \sup_{\theta\in K}\E_\theta\left[\left|-\frac{1}{N}\partial_{\beta}\p_{v_k} \ell_N(\theta)\right|^2\right] 
= \frac1N \sup_{\theta\in K}\E_\theta\left[\left|-\frac{1}{\sqrt{N}}\partial_{\beta}\p_{v_k} \ell_N(\theta)\right|^2\right]  \nn\\
&
\rev{
\lesssim \frac{1}{N} \sup_{\theta\in K} \E_{\theta}\left[\frac{1}{N}\sumi|X_i^{\top}Q_i(v)^{-1}(\p_{v_k}Q_i(v))Q_i(v)^{-1}|^2  |Y_i-X_i\beta|^2\right] 
}
\nn \\
&
\rev{
\lesssim \frac1N \frac1N \sumi \sup_{\theta\in K}\E_\theta[|Y_i-X_i\beta|^2]
}
\nn\\
&
\lesssim \frac1N \sup_{\theta\in K} \sup_{i\ge 1} \trace(Q_i(v)) \lesssim \frac1N \to 0.
\label{hm:thm_p1}
\end{align}
It follows that
\begin{equation}
\sup_{\theta \in K}\pr_\theta\left[
\left|-\frac{1}{N}\p_\beta \p_{v_k} \ell_N\left(\theta\right)\right| > \ep
\right] \to 0.
\nonumber
\end{equation}

To manage the remaining $\partial_{v_j}\p_{v_k} \ell_N(\theta)$, we note that
\begin{align*}
	&\E_{\theta}\left[\left(Y_i-X_i\beta\right)^\top \partial_{v_j}\left\{Q_i(v)^{-1}\left(\partial_{v_k} Q_i(v)\right)Q_i(v)^{-1}\right\}\left(Y_i-X_i\beta\right)\right]\\
	&=\E_{\theta}\left[\mathrm{trace}\left\{\partial_{v_j}\left\{Q_i(v)^{-1}\left(\partial_{v_k} Q_i(v)\right)Q_i(v)^{-1}\right\}\left(Y_i-X_i\beta\right)\left(Y_i-X_i\beta\right)^\top \right\}\right]\\
	&=\mathrm{trace}\left\{\partial_{v_j}\left\{Q_i(v)^{-1}\left(\partial_{v_k} Q_i(v)\right)Q_i(v)^{-1}\right\}Q_i(v)\right\}.
\end{align*}
Noting the identities
\begin{align*}
	&\partial_{v_j}\left\{Q_i(v)^{-1}\left(\partial_{v_k}Q_i(v)\right)Q_i(v)^{-1}\right\}Q_i(v)  \\
	&{}\qquad = - Q_i(v)^{-1}\left(\partial_{v_j}Q_i(v)\right)Q_i(v)^{-1}\left(\partial_{v_k}Q_i(v)\right) \\
	&{}\qquad\qquad  + Q_i(v)^{-1}(\partial_{v_j}\p_{v_k}Q_i(v)) - Q_i(v)^{-1}\left(\partial_{v_k}Q_i(v)\right)Q_i(v)^{-1}\left(\partial_{v_j}Q_i(v)\right),
\nn\\
& \partial_{v_j}\left\{\mathrm{trace}\left(Q_i(v)^{-1}\left(\partial_{v_k} Q_i(v)\right)\right)\right\} \nn\\
	&{}\qquad = \mathrm{trace}\left\{-Q_i(v)^{-1}\left(\partial_{v_j}Q_i(v)\right)Q_i(v)^{-1}\left(\partial_{v_k}Q_i(v)\right)+Q_i(v)^{-1}\left(\partial_{v_j}\p_{v_k}Q_i(v)\right)\right\},
\end{align*}
we obtain
\begin{align*}
\E_{\theta}\left[\partial_{v_j}\partial_{v_k} \ell_N(\theta)\right]
	&=\frac{1}{2} \sumi \E_{\theta}\Big[\left(Y_i-X_i\beta\right)^\top \partial_{v_j}\left\{Q_i(v)^{-1}\left(\partial_{v_k} Q_i(v)\right)Q_i(v)^{-1}\right\}\left(Y_i-X_i\beta\right)
	\nn\\
	&{}\qquad 
	-\partial_{v_j}\left\{\mathrm{trace}\left(Q_i(v)^{-1}\left(\partial_{v_k} Q_i(v)\right)\right)\right\}\Big] \nn\\
	&=-\frac{1}{2}\sumi\mathrm{trace}\left\{Q_i(v)^{-1}\left(\partial_{v_k}Q_i(v)\right)Q_i(v)^{-1}\left(\partial_{v_j}Q_i(v)\right)\right\}\end{align*}
for each $\theta$.
This together with \eqref{hm:thm_A2} and a similar estimate to \eqref{hm:thm_p1} concludes that
\begin{align*}
\sup_{\theta \in K}\pr_\theta\left[
\left|-\frac{1}{N}\p_{v_j} \p_{v_k} \ell_N\left(\theta\right) - U_{jk}(v)\right| > \ep
\right] &\leq 
\sup_{\theta \in K}\pr_\theta\left[
\left|-\frac{1}{N}\p_{v_j} \p_{v_k} \ell_N\left(\theta\right) + \frac{1}{N}\E_{\theta}\left[\partial_{v_j}\partial_{v_k} \ell_N(\theta)\right]\right| > \ep
\right] \\
&{}\qquad +
\sup_{\theta \in K}\pr_\theta\left[
\left|-\frac{1}{N}\E_{\theta}\left[\partial_{v_j}\partial_{v_k} \ell_N(\theta)\right] - U_{jk}(v)\right| > \ep
\right] \\
&\to 0.
\end{align*}
The proof of \eqref{C1} is complete.

\medskip

Turning to the asymptotic negligibility \eqref{C2(ii)}, we note the following expressions for the third-order derivatives:
\begin{align}
\partial^3_{\beta} \ell_N(\theta) &= 0,
\nn\\
\partial_{\beta}^2\p_{v_k} \ell_N(\theta) &= \sum_{i=1}^N X_i^\top Q_i(v)^{-1}\left(\partial_{v_k}Q_i\right)Q_i(v)^{-1}X_i,
\nn\\
\partial_{\beta}\p_{v_k}\p_{v_j}\ell_N(\theta) &= -\sum_{i=1}^N X_i^\top \partial_{v_k}\left\{Q_i(v)^{-1}\left(\partial_{v_j}Q_i(v)\right)Q_i(v)^{-1}\right\}\left(Y_i-X_i\beta\right),
\nn\\
\partial_{v_k}\p_{v_j}\p_{v_h}\ell_N(\theta) &= \frac{1}{2}\sum_{i=1}^N
\Big\{(Y_i-X_i\beta)^\top \partial_{v_k}\p_{v_j}\left\{Q_i(v)^{-1}\left(\partial_{v_h}Q_i(v)\right)Q_i(v)^{-1}\right\}\left(Y_i-X_i\beta\right)
\nn\\
&{}\qquad 
-\partial_{v_k}\p_{v_j}\left\{\mathrm{trace}(Q_i(v)^{-1}\partial_{v_h}Q_i(v))\right\} \Big\}.
\nonumber
\end{align}
For each $\theta$, we may and do focus on the set $\overline{B_\del(\theta)} \subset\Theta$, the closed ball at the center $\theta$ with radius $\del>0$ being small enough.
By the above expressions for $\p_{\theta}\mci_N(\theta)$, we have
\begin{align}
& \frac{1}{\sqrt{N}}\sup_{\theta'\in\Theta:\,|\theta'-\theta|\le c N^{-1/2}}\left| \p_{\theta}\mci_N(\theta') \right|
\nn\\
&\lesssim \frac{1}{\sqrt{N}} \sup_{\theta'\in\Theta:\,|\theta'-\theta|\le c N^{-1/2}} 
\frac1N\sumi \left(|Y_i-X_i\beta'|^2 + 1\right) 
\lesssim \frac{1}{\sqrt{N}} \frac1N\sumi \left(|Y_i-X_i\beta|^2 + 1\right).
\nonumber
\end{align}
The $\E_\theta$-expectation of the leftmost side can be bounded by a constant multiple of $N^{-1/2}$ uniformly in $\theta\in K$, concluding that $S_{2,N}(\ep,c,K) \to 0$.
\end{proof}


\section{Numerical experiments}
\label{hm:sec_simulations}
To evaluate the bias and the asymptotic normality of MLEs for the Gaussian linear mixed-effects IOU model, we conducted some simulation studies under two dataset structures following \cite{HugKenSteTil17}: balanced and unbalanced longitudinal data.
On the one hand, the balanced dataset consists of the subject's data where the numbers of time points and measurement time points are equal across all subjects.
On the other hand, for the unbalanced dataset, we allow that the number of time points per subject and the time intervals between consecutive time points need not be equal between subjects and within a subject.

For each Monte Carlo simulation, we generated longitudinal data $\{Y_i(t_{ij})\}_{j = 1}^{n_i}$ for $i = 1, \dots, N$ from the IOU model $\eqref{hm:model}$:
\begin{align*}
Y_i(t_{ij}) = X_i(t_{ij})^\top \beta_0 +Z_i(t_{ij})^\top b_i+ W_i(t_{ij}) +\epsilon_i(t_{ij}),
\nn
\end{align*}
where the ingredients are given as follows.
\begin{itemize}
\item The explanatory variables $X_i(t_{ij}) = (x_1(t_{ij}), x_2(t_{ij})) \in \mbbr^2$, $Z_i(t_{ij}) = (z_1(t_{ij}), z_2(t_{ij})) \in \mbbr^2$ were generated as $x_1(t_{ij}) = t_{ij}$, $x_2(t_{ij}) = 0 \, \text{or} \, 1$ according as the Bernoulli distribution with the parameter $0.5$ before starting Monte Carlo simulation, and $(z_1(t_{ij}), z_2(t_{ij})) = (1, t_{ij})$.
\item The random effect vector $\ds{b_i \sim N_2\left(0,\, \begin{pmatrix}
\gam_{0, 1}^2 & \gam_{0, 2} \\
\gam_{0, 2} & \gam_{0, 3}^2
\end{pmatrix}
\right)}$.

\item The system noise vector $(W_i(t_{ij}))_{j = 1}^{n_i} \sim N_{n_i}(0, H_i(\alpha_0, \tau_0))$.
\item The measurement error vector $(\epsilon_i(t_{ij}))_{j = 1}^{n_i} \sim N_{n_i}(0, \sigma_0^2 I_{n_i})$.
\end{itemize}
The true parameter was given as 
\begin{align*}
(\beta_0, v_0) = (\beta_{0, 1}, \beta_{0, 2}, \gam_{0, 1}, \gam_{0, 2}, \gam_{0, 3}, \alpha_0, \tau_0, \sigma_0) = (-0.25, 0.50, 1.25, 1.00, 1.50, 1.30, 0.40, 1.25). \nn
\end{align*}
The number of time points $n_i$ and measurement time points $\{t_{ij}\}_{j = 1}^{n_i}$ for $i = 1, \dots, N$ were set differently for balanced and unbalanced longitudinal data simulation:
\begin{itemize}
\item For the balanced data simulation, we set the number of time points as $n_i = 20$ and time points $t_{ij} = j$ for all $i = 1, \dots, N$, that is, the time intervals between consecutive time points are equal between subjects and within a subject;
\item For the unbalanced data simulation, we generated the data under the setting that the number of time points $n_i$ was obtained from the integer part of $\textrm{Uniform}(15, 20)$-random number and measurement time points $t_{i1}, \dots, t_{in_i}$ were randomly selected from $\{1,2,\dots,20\}$ before starting the simulation. 
\end{itemize}
We generated 1000 datasets for all the Monte Carlo simulations, and we set the sample size $N = 250$ or $500$ for the balanced and unbalanced longitudinal datasets, respectively. To optimize the log-likelihood function \eqref{hm:log-LF}, we used the built-in $\texttt{optim}$ function in the R software. For optimizations in all simulations, we used an $8$-dimensional vector of values $1$ as the initial value. We used in all optimizations the Nelder-Mead method as an optimization algorithm because of a complication for the first and second derivative functions of the log-likelihood function. 

Tables $\ref{ti:balanced table}$ and $\ref{ti:unbalanced table}$ show the bias and the standard error for each parameter and true parameter, calculated by the Monte Carlo method.
To estimate the inaccuracy of Monte Carlo samples, we introduce the Monte Carlo standard error (MCSE, e.g. \cite{MWC2019}) deified by
\begin{align*}
\text{MCSE} = \sqrt{\frac{1}{M (M - 1)} \sum_{m = 1}^M (\hat{\theta}_m - \bar{\theta})^2}, \nn
\end{align*}
where $M$ denotes the number of iterations of the simulation, $\hat{\theta}_m$ is the estimate of $\theta$ for $m$th repetition, and $\bar{\theta}$ is the sample mean of $\hat{\theta}_m$ across repetitions. 

As shown in Table $\ref{ti:balanced table}$ and Table $\ref{ti:unbalanced table}$, there was little difference between the biases of all parameters in both two dataset structures and sample size settings ($N = 250, 500$). The estimates of the fixed-effect parameters and the variance parameter of the measurement error were unbiased for all simulations. The biases of the variance parameters for the random effects were not large to matter. In contrast, the biases of the variance parameters for the system noise were not negligibly small. One possible reason is that, as can be seen from Figure $\ref{ti:LLplot_unbalanced_N500}$, the optimizations were not successful because of the small curvature around the true value of the log-likelihood function for $(\alpha, \tau)$. The previous study \cite{HugKenSteTil17} recommends a reparametrized Gaussian mixed-effects IOU model as $\alpha$ and $\omega$ ($\omega := \tau^2/\alpha^2$); however, in our simulation studies, the calculated $\omega$ had a large bias. 

Figures $\ref{ti:hist_unbalanced_N500}$ and $\ref{ti:qq_unbalanced_N500}$ show histograms of the Studentized MLEs and normal quantile-quantile plots (Q-Q plots) under the unbalanced longitudinal data setting ($N = 500$), respectively. From these figures, the standard normal approximation seemed to hold for all the MLEs except $\hat{\sigma}_N$. The magnitude of the variance parameter of the measurement error was very small. 

The problem we faced in our numerical experiments was the computational cost of obtaining the MLEs. For example, the average time was about 7 minutes for 1 iteration in the balanced data simulation with $N = 500$.
One possible solution to this problem of computation time is to change the optimization method. The previous study \cite{HugKenSteTil17} recommends using the Newton-Raphson (NR) algorithm in terms of convergence and the time taken to reach convergence. Considering the actual application of this model, it would be better to use the NR method with a low computational cost. In the present study, we do not go into details in this direction.

\begin{table}[h]
\centering
\caption{Bias and MCSE for each MLE in the balanced longitudinal dataset structure}
\label{ti:balanced table}
\vspace{5pt}
\scalebox{0.7}{
\begin{tabular}{cccccccccc}
\toprule
\multicolumn{1}{c}{$N$} & \multicolumn{8}{c}{Bias (MCSE) for each parameter}\\
\cmidrule(lr){2-9}
\multicolumn{1}{c}{} & \multicolumn{1}{c}{$\beta_1$} & \multicolumn{1}{c}{$\beta_2$} & \multicolumn{1}{c}{$\gam_1$} & \multicolumn{1}{c}{$\gam_2$} & \multicolumn{1}{c}{$\gam_3$} & \multicolumn{1}{c}{$\alpha$} & \multicolumn{1}{c}{$\tau$} & \multicolumn{1}{c}{$\sig$} \\

\hline 
\hline

250 & -0.0006 (0.0028) & -0.0026 (0.0043) & -0.0170 (0.0027) & -0.0358 (0.0052) & -0.0111 (0.0022) & 0.8481 (0.0121) & 0.2459 (0.0040) & -0.0026 (0.0005) \\

500 & 0.0008 (0.0022) & -0.0111 (0.0032) & -0.0208 (0.0021) & -0.0565 (0.0044) &  -0.0156 (0.0017) & 0.9095 (0.0102) & 0.2686 (0.0035) & -0.0037 (0.0004)\\
\hline
\end{tabular}
}
\end{table}

\begin{table}[h]
\centering
\caption{Bias and MCSE for each MLE in the unbalanced longitudinal dataset structure}
\label{ti:unbalanced table}
\vspace{5pt}
\scalebox{0.7}{
\begin{tabular}{cccccccccc}
\toprule
\multicolumn{1}{c}{$N$} & \multicolumn{8}{c}{Bias (MCSE) for each parameter}\\
\cmidrule(lr){2-9}
\multicolumn{1}{c}{} & \multicolumn{1}{c}{$\beta_1$} & \multicolumn{1}{c}{$\beta_2$} & \multicolumn{1}{c}{$\gam_1$} & \multicolumn{1}{c}{$\gam_2$} & \multicolumn{1}{c}{$\gam_3$} & \multicolumn{1}{c}{$\alpha$} & \multicolumn{1}{c}{$\tau$} & \multicolumn{1}{c}{$\sig$} \\

\hline 
\hline

250 & 0.0026 (0.0030) & -0.0023 (0.0043) & -0.0195 (0.0028) & -0.0398 (0.0054) & -0.0127 (0.0023) & 0.8182 (0.0120) & 0.2370 (0.0040) & -0.0027 (0.0005) \\

500 & 0.0036 (0.0022) & 0.0019 (0.0033) & -0.0219 (0.0021) & -0.0634 (0.0046) &  -0.0157 (0.0018) & 0.8854 (0.0107) & 0.2610 (0.0036) & -0.0034 (0.0004)\\
\hline
\end{tabular}
}
\end{table}

\begin{figure}[h]
	\begin{center}
			\includegraphics[clip,width=12cm,height=8cm]{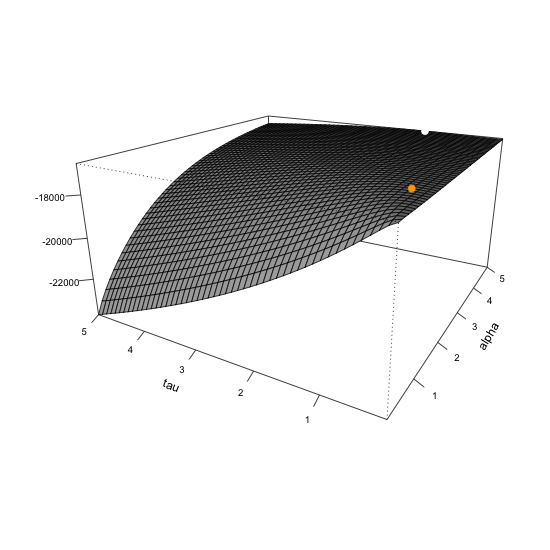}
			\caption{Curved surface of log-likelihood function for $(\alpha, \tau)$ in the case of sample size $N = 500$ and time points $15 \leq n_i \leq 20$; the remaining six parameters are set to be their true values. The orange point is true parameter $(\alpha_0, \tau_0)$. The white point is the calculated MLE.}
			\label{ti:LLplot_unbalanced_N500}
	\end{center}
\end{figure}

\begin{figure}[h]
	\begin{center}
			\includegraphics[clip,width=12cm,height=12cm]{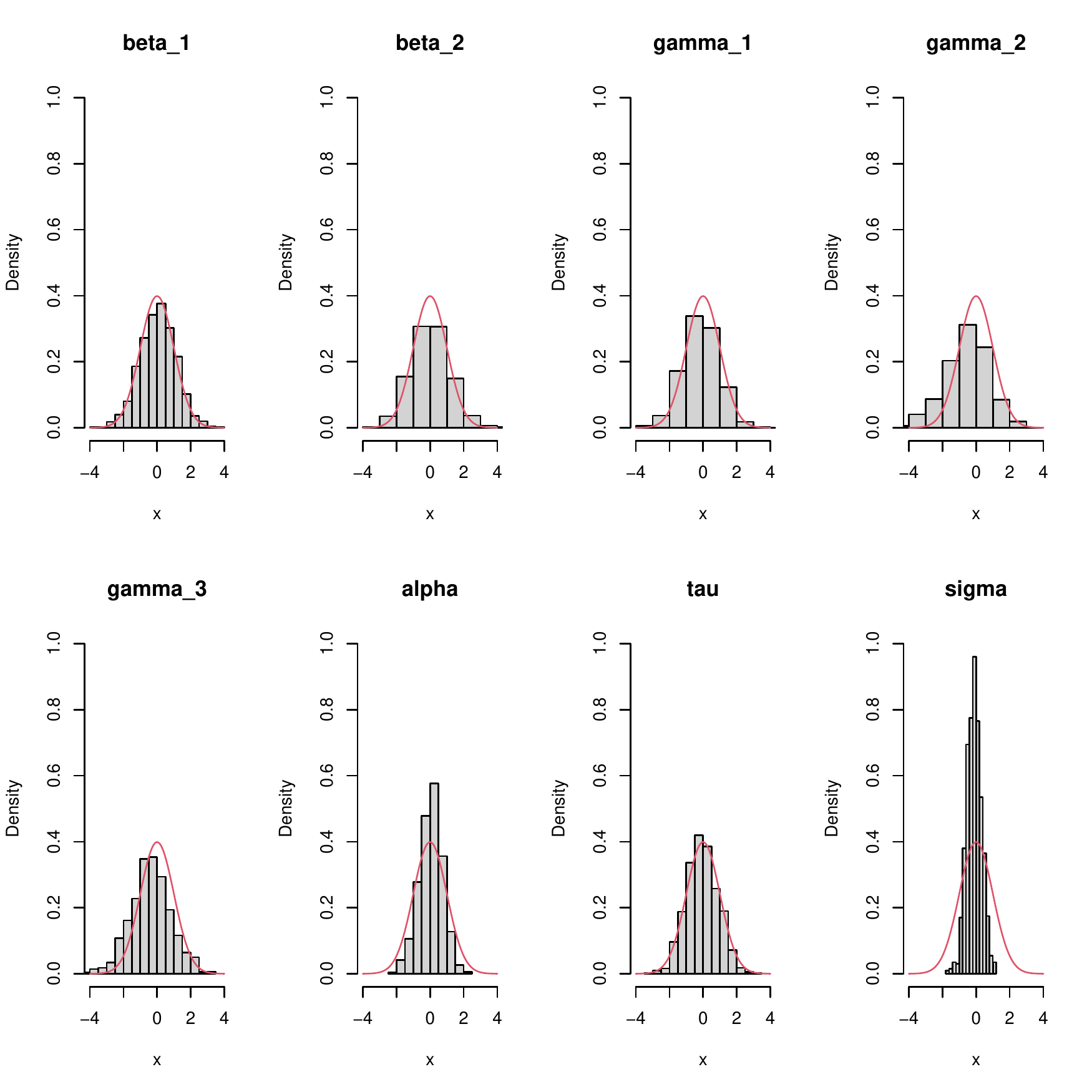}
			\caption{Histograms of the Studentized MLEs for unbalanced longitudinal data ($N$ = 500) and probability density function of Gaussian distribution (red curve).}
			\label{ti:hist_unbalanced_N500}
	\end{center}
\end{figure}

\begin{figure}[h]
	\begin{center}
			\includegraphics[clip,width=12cm,height=12cm]{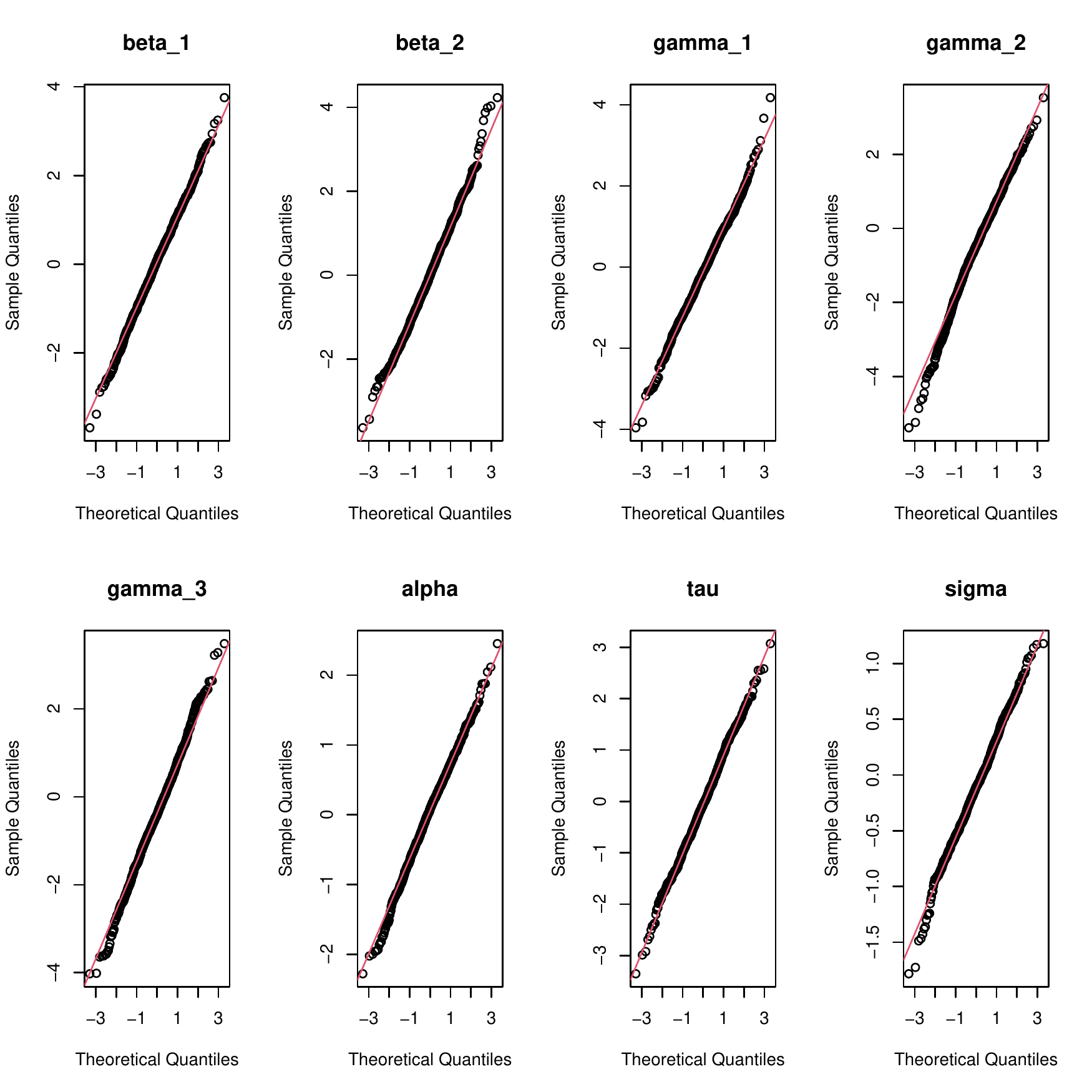}
			\caption{Normal Q-Q Plots of the Studentized MLEs for unbalanced longitudinal data ($N$ = 500).}
			\label{ti:qq_unbalanced_N500}
	\end{center}
\end{figure}

\section{Discussion}
\label{ti:sec_discuss}
The Gaussian mixed-effects IOU model is useful in terms of easy interpretations of serial-correlation structure in each individual. Furthermore, this model can be applied naturally to longitudinal data in which the measurement intervals differ between and within individuals. In this paper, we studied derived the local likelihood asymptotics: the MLE has the asymptomatic normality and the asymptotic efficiency (Theorem \ref{hm:thm_local.LF.asymp} and Remark \ref{hm:rem_LAN}).
Our results will underlie practical developments of this model. 
For this model to be more widely utilized in actual applications, as already mentioned in Remarks \ref{hm:rem_future.work} and \ref{hm:rem_model.extension}, we need further theoretical developments including robustification against distributional misspecification and model selection criteria.

\bigskip

\noindent
{\bf Acknowledgement.}
The authors thank the reviewer for the helpful comments.
This work was partially supported by JST CREST Grant Number JPMJCR2115, Japan, and by JSPS KAKENHI Grant Number 22H01139 (HM).

\bigskip

\bibliographystyle{abbrv} 

\begin{thebibliography}{}

\end{thebibliography}


\begin{thebibliography}{10}

\bibitem{BasSco83}
I.~V. Basawa and D.~J. Scott.
\newblock {\em Asymptotic optimal inference for nonergodic models}, volume~17
  of {\em Lecture Notes in Statistics}.
\newblock Springer-Verlag, New York, 1983.

\bibitem{BosTayLaw1998}
W. J. Boscardin, J. M. Taylor, and N. Law.
\newblock Longitudinal models for AIDS marker data.
\newblock {\em Statistical methods in medical research}, 7(1): 13-27, 1998.

\bibitem{Dig88}
P.~J. Diggle.
\newblock An approach to the analysis of repeated measurements.
\newblock {\em Biometrics}, 44(4):959--971, 1988.

\bibitem{GSACZS22}
E.~Gecili, S.~Sivaganesan, O.~Asar, J.~P. Clancy, A.~Ziady, and R.~D.
  Szczesniak.
\newblock Bayesian regularization for a nonstationary gaussian linear mixed
  effects model.
\newblock {\em Statistics in Medicine}, 41(4):681--697, 2022.

\bibitem{hughes2017_stata}
R.~A. Hughes, M.~G. Kenward, J.~A. Sterne, and K.~Tilling.
\newblock Analyzing repeated measurements while accounting for derivative
  tracking, varying within-subject variance, and autocorrelation: The
  xtmixediou command.
\newblock {\em The Stata Journal}, 17(3):573--599, 2017.

\bibitem{HugKenSteTil17}
R.~A. Hughes, M.~G. Kenward, J.~A.~C. Sterne, and K.~Tilling.
\newblock Estimation of the linear mixed integrated {O}rnstein-{U}hlenbeck
  model.
\newblock {\em J. Stat. Comput. Simul.}, 87(8):1541--1558, 2017.

\bibitem{Jeg82}
P.~Jeganathan.
\newblock On the asymptotic theory of estimation when the limit of the
  log-likelihood ratios is mixed normal.
\newblock {\em Sankhy\=a Ser. A}, 44(2):173--212, 1982.

\bibitem{MWC2019}
T.~P. Morris, I.~R. White, and M.~J. Crowther.
\newblock Using simulation studies to evaluate statistical methods.
\newblock {\em Statistics in Medicine}, 38(11):2074--2102, 2019.

\bibitem{fBm_ME}
O.~T. Stirrup, A.~G. Babiker, J.~R. Carpenter, and A.~J. Copas.
\newblock Fractional {B}rownian motion and multivariate-t models for
  longitudinal biomedical data, with application to {CD}4 counts in
  {HIV}-positive patients.
\newblock {\em Stat. Med.}, 35(9):1514--1532, 2016.

\bibitem{Swe80}
T.~J. Sweeting.
\newblock Uniform asymptotic normality of the maximum likelihood estimator.
\newblock {\em Ann. Statist.}, 8(6):1375--1381, 1980.
\newblock Corrections: (1982) {\it Annals of Statistics} {\bf 10}, 320.

\bibitem{SyTayCum97}
J.~P. Sy, J.~M.~G. Taylor, and W.~G. Cumberland.
\newblock A {S}tochastic {M}odel for the {A}nalysis of {B}ivariate
  {L}ongitudinal {AIDS} {D}ata.
\newblock {\em Biometrics}, 53(2):542--555, 1997.

\bibitem{Taj21_m.thesis}
H.~Tajima.
\newblock Asymptotic inference for linear mixed-effects model with integrated
  ornstein-uhlenbeck process and information criterion.
\newblock Master thesis, Kyushu University, 2021.

\bibitem{TayCumSy94}
J.~M.~G. Taylor, W.~G. Cumberland, and J.~P. Sy.
\newblock A stochastic model for analysis of longitudinal aids data.
\newblock {\em Journal of the American Statistical Association},
  89(427):727--736, 1994.

\bibitem{ZhaLinRazSow98}
D.~Zhang, X.~Lin, J.~Raz, and M.~Sowers.
\newblock Semiparametric stochastic mixed models for longitudinal data.
\newblock {\em Journal of the American Statistical Association},
  93(442):710--719, 1998.

\end{thebibliography}

\end{document}